\documentclass[12pt,a4paper]{article}

\usepackage[utf8]{inputenc}
\usepackage{amsfonts, amsmath, amssymb, amsthm, mathrsfs}
\usepackage{geometry, graphicx}
\usepackage[hidelinks]{hyperref}
\usepackage{cleveref}
\usepackage{enumitem}
\hypersetup{colorlinks={true},linkcolor={blue}}

\usepackage{lipsum}
\usepackage{tikz-cd}

\newtheorem{theorem}{Theorem}
\newtheorem{corollary}{Corollary}
\newtheorem{defi}{Definition}
\newtheorem{lem}{Lemma}
\newtheorem{sublem}{Lemma}[lem]
\newtheorem{rmk}{Remark}
\newtheorem{prop}{Proposition}

\usepackage{setspace,caption}
\usepackage{titlesec}
\titleformat{\section}[block]{\bfseries}{\thesection.}{1em}{}

\newcommand\genby[1]{\langle#1\rangle}

\newcommand\mods{\text{ mod }}

\title{Ideals in the Goldman Algebra}
\author{Minh Nguyen}

\begin{document}
\maketitle
\begin{abstract}
The goal of this work is to study the ideals of the Goldman Lie algebra $S$. To do so, we construct an algebra homomorphism from $S$ to a simpler algebraic structure, and focus on finding ideals of this new structure instead. The structure $S$ can be regarded as either a $\mathbb{Q}$-module or a $\mathbb{Q}$-module generated by free homotopy classes. For $\mathbb{Z}$-module case, we proved that there is an infinite class of ideals of $S$ that contain a certain finite set of free homotopy classes. For $\mathbb{Q}$-module case, we can classify all the ideals of the new structure and consequently obtain a new class of ideals of the original structure. Finally, we show an interesting infinite chain of ideals that are not those ideals obtained by considering the new structure.
\end{abstract}

\setlength{\parskip}{0pt}

\section{Lie algebra homomorphism from $\mathbb{Z}$-module of freely homotopic classes on compact orientable surface}
\begin{defi}
Let $\hat{\pi_0}$ be the set of freely homotopic classes of curves on the surface and $\mathbb{Z}[\hat{\pi_0}]$ be the $\mathbb{Z}$-module generated by $\hat{\pi_0}$.

Recall that for closed surface with genus $g$, its fundamental group is: 

\[\pi_1 = \genby{a_1, a_2,\dotsc, a_{2g-1},a_{2g}\ |\ a_1a_2a_1^{-1}a_2^{-1}\dotsc a_{2g-1}a_{2g}a_{2g-1}^{-1}a_{2g}^{-1} = 1}\]

For surface with boundary, its fundamental group is a free group generated by $n$ generators: $\pi_1 = \genby{a_1, a_2, \dotsc, a_n}$

As a result, we can represent each element of $\pi_1$ uniquely as a \textbf{linear reduced word} in letter $a_i$, the word with letters $a_i$ so that its $2$ consecutive letters do not cancel each other.

Now define \textbf{cyclic word} as a word whose letters are arranged on a circle instead of a line, and a reduced cyclic word is a cyclic word whose 2 consecutive letters do not cancel each other.

A \textbf{linear representative} $W$ of the cyclic word $\mathscr{W}$ is the linear word obtained by making a cut to $2$ consecutive word of $\mathscr{W}$. From here, we can see that $\mathscr{W}$, and its linear representative, $W$, have the same word length.

We have the $1-1$ correspondence between $\hat{\pi_0}$ and the conjugacy classes of $\pi_1$ so $x \in \hat{\pi_0}$ will correpond to a conjugacy class $\gamma_x$ of $\pi_1$. Therefore, we can represent each element $x$ of $\hat{\pi_0}$ uniquely as a cyclic reduced word $\mathscr{W}$ so that if $W$ is a reduced linear word correponding to an element of $\pi_1$ in the conjugacy class $\gamma_x$, then by repeatedly canceling the first and last letters of $W$, we will get a linear representative of $\mathscr{W}$.
\end{defi}

\begin{defi}\label{def2}
For closed compact surface $S$ with genus $g$ and with fundamental group
$\pi_1 = \genby{a_1, a_2, \dotsc ,a_{2g-1},a_{2g}|a_1a_2a_1^{-1}a_2^{-1}\dotsc a_{2g-1}a_{2g}a_{2g-1}^{-1}a_{2g}^{-1} = 1}$, we define the abelianization of its fundamental group, $A(n)$, as $\genby{a_1, a_2,\dotsc, a_{n-1},a_{n}|a_ia_j = a_ja_i}$

For compact surface with boundary $S$ with the fundamental group \\$\pi_1 = \genby{a_1, a_2, \dotsc, a_n}$, we define the abelianization of its fundamental group, $A(n)$, as $\genby{a_1, a_2,\dotsc, a_{n-1},a_{n}|a_ia_j = a_ja_i}$ instead.

\vspace{2mm}
Now let $\mathbb{Z}[A(n)]$ be the $\mathbb{Z}$-module generated by $A(n)$, and consider the map $\mathbf{Ab}$
\begin{align*}
\mathbf{Ab}: \mathbb{Z}[\hat{\pi_0}] &\to \mathbb{Z}[A(n)]\\
\sum\limits_{i = 1}^k q_iu_i & \mapsto \sum\limits_{i = 1}^k q_i\mathbf{Re}(w_i) 
\end{align*}

Here $q_i \in \mathbb{Z}$, $u_i \in \hat{\pi_0}$. Also, $w_i$ is one chosen reduced linear word so that it corresponds to an element $v_i \in \pi_1$ which is in the conjugacy class of $\pi_1$ that corresponds to $u_i$. We call this $w_i$ the $\textbf{linear element}$ of $u_i$. Note that there may be multiple linear elements of the same $u_i \in \hat{\pi_0}$, and each linear representative of the reduced cyclic word corresponding to $u_i$ is a linear element of $u_i$. Moreover, $\mathbf{Re}(w) = \prod\limits_{i=1}^g a_i^{m_i}$ is defined as an element in $A(n)$ that is obtained by reordering the letter $a_i$ in the reduced linear word $w$.

Now $\mathbf{Ab}$ doesn't depend on the choice of linear element $w_i$ of $u_i$. Indeed, if $w_i$ and $w'_i$ are two different linear elements of $u_i$, then $w_i$ and $w'_i$ must be conjugate to each other, i.e, there is some linear word $w$ so that $w_i = ww'_iw^{-1}$, and, therefore $\mathbf{Re}(w_i) = \mathbf{Re}(w'_i)$.

Finally, it is easy to see that $\mathbf{Ab}$ is a $\mathbb{Z}$-module homomorphism and is surjective.
\end{defi}

\begin{defi}\label{def3}
For any $u_1$ and $u_2 \in \hat{\pi_0}$, let $m[w_1, w_2]$ be the sum of integer coefficients of their Goldman bracket. This means that if we take some represetative curves $\alpha$ and $\beta$ of $u_1$ and $u_2$, then $m[u_1, u_2] = \sum\limits_{p \in \alpha \cap \beta} sign(p)$ , the sum of signs over all intersection points $p$ of curves $\alpha$ and $\beta$. 

This definition will not depend on $\alpha$ and $\beta$ because Goldman bracket is the same for all curves in some same freely homotopic class of curves. 
\end{defi}

\begin{defi}\label{def4}
We now define a Lie bracket on $\mathbb{Z}[A(n)]$ to make it a Lie algebra. 

Since $a_i$ and $a_j$ has word-length 1, their preimages from the map $\mathbf{Ab}$ is uniquely defined. As a result, we can define an symplectic product for $a_i$ and $a_j$:
\begin{equation*}
\genby{a_i, a_j} = m[\mathbf{Ab}^{-1}(a_i), \mathbf{Ab}^{-1}(a_j)] \ \forall 1\leq i, j \leq n.
\end{equation*}
, $m$ is defined in \cref{def3}

Notice that $\genby{a_i, a_j} = 1$ or $0$ or $-1$ since there is at most 1 linking pair and $\genby{a_i, a_j} = -\genby{a_j, a_i}$ because of the anti-symmetric property of Goldman bracket.\\

Now for $w_1 = \prod\limits_{i=1}^g a_i^{m_i}, w_2 = \prod\limits_{i=1}^g a_i^{n_i} \in A(n)$, we define their sympletic product:
\begin{equation*}
\genby{w_1, w_2}  = \sum\limits_{i,j=1}^{n} m_in_j\genby{a_i, a_j}
\end{equation*}

Therefore,
$\genby{wv, ts} = \genby{w, t} + \genby{v, t} + \genby{w, s} + \genby{v, s} \ \forall w, v, t, s \in A(n)$

\vspace{4mm}
Now we define the bracket on elements of $A(n)$:
\begin{equation*}
[w_1, w_2] = \genby{w_1, w_2} w_1 w_2 = \genby{w_1, w_2}\prod\limits_{i = 1}^g a_i ^{m_i + n_i}\ \forall w_1, w_2 \in A(n)
\end{equation*}

Now the bracket in $A(n)$ is extended to the bracket in $\mathbb{Z}[A(n)]$ by $\mathbb{Z}$-bilinearity.

In particular, for any $\sum\limits_{i = 1}^k k_iw_i$ and $\sum\limits_{j = 1}^l l_jv_j \in \mathbb{Z}[A(n)]$ with $k_i, l_j \in \mathbb{Z}$, and $w_i, v_j \in A_n$, 
\begin{equation*}
[\sum\limits_{i = 1}^k k_iw_i, \sum\limits_{j = 1}^l l_jv_j] = \sum\limits_{\substack{1 \leq i \leq k \\ 1 \leq j \leq l}} k_il_j [w_i, v_j]
\end{equation*}

It is not difficult to check the anti-symmetry and Jacobi identities for this bracket on elements in $A(n)$ and, therefore, on elements in $\mathbb{Z}[A(n)]$. This makes $\mathbb{Z}[A(n)]$ become a Lie algebra.
\end{defi}

\vspace{2mm}

\begin{theorem}\label{thm1}
We prove that $m[u_1, u_2] = \genby{\mathbf{Ab}(u_1), \mathbf{Ab}(u_2)}\ \forall u_1, u_2 \in \hat{\pi_0}$. As a result, the map $\mathbf{Ab}$ is a Lie algebra homomorphism.
\end{theorem}
\begin{proof}
Let $l(c)$ be the length of the linear or cyclic word $c$.
 
We prove \cref{thm1} by induction on the maximum $m = \max\{l(c_1), l(c_2)\}$, where $c_1$ and $c_2$ are cyclic words that correspond to $u_1$ and $u_2$. 

If the maximum word-length $m$ is $1$, then \cref{thm1} follows from the definition of symplectic product on $A(n)$.  Assume that \cref{thm1} is true for all $m \leq k -1$. Now we prove that it is also true for $m = k$.

Now suppose $w_1$ and $w_2$ are linear representatives of the cyclic reduced words correponding to $u_1$ and $u_2$.

Now suppose $w_1 = x_1y_1$ ($x_1$ is some letter $a_i$), and $w_2 = x_2y_2$ ($x_2$ is also some letter $a_j$). 

So $l(y_i) \leq k - 1$, where $y_i$ are some linear reduced word. 

Choose the base point $P$ on the surface for the fundamental group $\pi_1$. For $i = \overline{1, 2}$, choose a loop $\beta_i$ based at $P$ that has the corresponding reduced linear word $x_i$ in the fundamental group, and a loop $\gamma_i$ that has the reduced linear word $y_i$ in fundamental group. Let $\alpha_i$ be the loop obtained by concatenating $\beta_i$ with $\gamma_i$. Then $\alpha_i$ will have the linear reduced word $w_i$ in the fundamental group, and therefore, $\alpha_i$ is in the freely homotopic class $u_i \in \hat{\pi_0}$. 

Now we change the loop $\alpha_2$ to $\alpha'_2$, which is freely homotopic to $\alpha_2$, by slightly moving the point $P \in \alpha_2$ to a different point $P' \in \alpha_2'$ so that the two new sub-loops $\beta'_2$ and $\gamma'_2$ of $\alpha'_2$, which intersect at $P'$ and make up $\alpha'_2$, are homotopic to two orgininal sub-loops $\beta_2$ and $\gamma_2$. 

So we have $\beta_1$ and $\gamma_1$ are in some freely homotopic class $u_{\beta_1}$ and $u_{\gamma_1} \in \hat{\pi_0}$.

$\beta_2$ and $\beta'_2$ are in some freely homotopic class $u_{\beta_2} \in \hat{\pi_0}$, and $\gamma_2$ and $\gamma'_2$ are in some freely homotopic class $u_{\gamma_2} \in \hat{\pi_0}$.

Note that $x_1, y_1, x_2, y_2$ are linear elements of $u_{\beta_1}, u_{\gamma_1}, u_{\beta_2}, u_{\gamma_2}$ respectively. Moreover, lengths of cyclic words or linear elements of $u_{\beta_i}$ and $u_{\gamma_i}$ for $i = \overline{1, 2}$ are at most $k - 1$.

Also note that $P \neq P'$ so the set of intersection points of $\alpha_1$ and $\alpha'_2$ is the disjoint union of sets of intersection points of $\beta_1$ and $\beta'_2$, $\beta_1$ and $\gamma'_2$, $\gamma_1$ and $\beta'_2$, $\gamma_1$ and $\gamma'_2$. Therefore,
\begin{equation*}
\begin{split}
m[u_1, u_2] & = m[\alpha_1, \alpha_2] = m[\alpha_1, \alpha'_2] = \sum\limits_{p \in \alpha_1 \cap \alpha'_2}sign(p) \\
& = \sum\limits_{p \in \beta_1 \cap \beta'_2}sign(p) + \sum\limits_{p \in \beta_1 \cap \gamma'_2}sign(p) + \sum\limits_{p \in \gamma_1 \cap \beta'_2}sign(p) + \sum\limits_{p \in \gamma_1 \cap \gamma'_2}sign(p)  \\
& = m[u_{\beta_1}, u_{\beta_2}] + m[u_{\beta_1}, u_{\gamma_2}] + m[u_{\gamma_1}, u_{\beta_2}] + m[u_{\gamma_1}, u_{\gamma_2}]
\end{split}
\end{equation*}
By induction hypothesis on $4$ pairs of elements in $\hat{\pi_0}$, $(u_{\beta_1}, u_{\beta_2}), (u_{\beta_1}, u_{\gamma_2}), (u_{\gamma_1}, u_{\beta_2})$, and $(u_{\gamma_1}, u_{\gamma_2})$, this sum $m[u_1, u_2]$ equals to:
\begin{equation*}
\begin{split}
& = \genby{\mathbf{Ab}(u_{\beta_1}), \mathbf{Ab}(u_{\beta_2})} + \genby{\mathbf{Ab}(u_{\beta_1}), \mathbf{Ab}(u_{\gamma_2})} \\
& \hspace{4cm} +  \genby{\mathbf{Ab}(u_{\gamma_1}), \mathbf{Ab}(u_{\beta_2})} +  \genby{\mathbf{Ab}(u_{\gamma_1}), \mathbf{Ab}(u_{\gamma_2})} \\
& = \genby{\mathbf{Re}(x_1), \mathbf{Re}(x_2)} + \genby{\mathbf{Re}(x_1), \mathbf{Re}(y_2)} + \genby{\mathbf{Re}(y_1), \mathbf{Re}(x_2)} + \genby{\mathbf{Re}(y_1), \mathbf{Re}(y_2)} \\
& = \genby{\mathbf{Re}(x_1).\mathbf{Re}(y_1), \mathbf{Re}(x_2).\mathbf{Re}(y_2)} \\
& = \genby{\mathbf{Re}(w_1), \mathbf{Re}(w_2)}\\
& = \genby{\mathbf{Ab}(u_1), \mathbf{Ab}(u_2)}
\end{split}
\end{equation*}
This finishes the induction step and also proof for \cref{thm1}.
\end{proof}

\begin{rmk}
Now we consider $\mathbb{Q}[\hat{\pi_0}]$ as a $\mathbb{Q}$-module generated by $\hat{\pi_0}$, and as a Lie algebra with Goldman bracket. We also consider $\mathbb{Q}[A(n)]$, the $\mathbb{Q}$-module generated by $A(n)$, instead, and define the map:
\begin{align*}
\mathbf{Ab_{\mathbb{Q}}}: \mathbb{Q}[\hat{\pi_0}] &\to \mathbb{Q}[A(n)]\\
\sum\limits_{i = 1}^k q_iu_i & \mapsto \sum\limits_{i = 1}^k q_i\mathbf{Re}(w_i) 
\end{align*}
in a similar way that we defined $\mathbf{Ab}$ in \cref{def2}, but with $q_i \in \mathbb{Q}$ instead.

Moreover, we can put a Lie bracket on $\mathbb{Q}[A(n)]$ and make it a Lie algebra by extending bracket on $A(n)$, which is defined in \cref{def4}, by $\mathbb{Q}$-bilinearity.

Then, by \cref{thm1}, $\mathbf{Ab}_{\mathbb{Q}}$ is also a Lie algebra homomorphism.
\end{rmk}

\section{A class of ideals obtained from preimage of the Lie algebra homomorphism $\mathbf{Ab}$}
\begin{defi}
Denote by $G$ the set of integer multiples of elements of $A(n)$ (defined in \cref{def2}), or $G= \{\alpha x, \text{ where } \alpha \in \mathbb{Z} \text{, and } x = \prod\limits_{j=1}^n a_j^{i_j} \in A(n),\ i_j \in \mathbb{Z} \}$. 

We say that a $\mathbb{Z}$-submodule  of $\mathbb{Z}[A(n)]$ is $\mathbf{geometric}$ if it is a $\mathbb{Z}$-submodule generated by a subset of $G$.

\vspace{2mm}
Consider a $\mathbb{Z}$-submodule $M$ of $\mathbb{Z}[A(n)]$.  For each $n$-tuple of integers $(i_1,\dotsc, i_n)$, if there is non-zero $\alpha_0$ such that $ \alpha_0 w$ is in $M$ , where $w = \prod\limits_{j=1}^n a_j^{i_j} \in A(n)$,  then we define $\alpha_{w, M} = \alpha_M(i_1,i_2,\dotsc,i_n)$ as the smallest positive integer $\alpha$ such that $\alpha w = \alpha \prod\limits_{j=1}^n a_j^{i_j} \in M$. If there is no such multiple, then define $\alpha_{w, M} = \alpha_M(i_1,i_2,\dotsc,i_n) = 0$.
\end{defi}

\begin{lem}\label{lm1}
A geometric $\mathbb{Z}$-submodule $I$ is an ideal of the Lie algebra $\mathbb{Z}[A(n)]$ if and only if $\alpha_w | \genby{v,w}\alpha_v$ for all $v,w \in A(n)$. Here $\alpha_w = \alpha_{w, I}$
\end{lem} 
\begin{proof}
Suppose there is a geometric $\mathbb{Z}$-submodule $I$ satisfying $\alpha_w | \genby{v,w}\alpha_v$ for all $v,w \in A(n)$. To prove that $I$ is an ideal, we just need to prove that $[\alpha_ww, v] \in I$ for every $v \in A(n)$

By lemma's assumption, $\alpha_{wv} | \genby{w, wv} \alpha_w = \genby{w, v} \alpha_w$, so by definition of $\alpha_{wv}$, we must have $[\alpha_ww, v] = \alpha_w \genby{w, v} wv \in I$.

\vspace{2mm}
The other direction of the theorem is proved in a similar way.
\end{proof}

\begin{theorem}
For each finite set $Y$ of some freely homotopic classes on surface $S$ , there is an infinite class of nontrivial ideals that contain $Y$. 
\end{theorem}

\begin{proof}
Assume that $\mathbb{Z}[A(n)]$ is the Lie algebra corresponding to the surface $S$. We just need to prove that there is a class of non-trivial ideals of $\mathbb{Z}[A(n)]$ that contain $X$ for any finite subset $X$ of $A(n)$. Then the preimage of these ideals under the Lie algebra homomorphism $\mathbf{Ab}$ will be the ideals that we look for. Here $X = \mathbf{Ab}(Y)$

Each element of $X$ is combination of some finite element in $A(n)$. 

Since $X$ is finite, $X$ must be the subset of the $\mathbb{Z}$-submodule generated by \\$X_1 = \{a_1^{i_1}a_2^{i_2} \dotsc a_n^{i_n}| (i_1, i_2,\dotsc, i_n) \in K_0 \}$, for some finite subset $K_0 \subset \mathbb{Z} ^n$.

\vspace{2mm}
Now consider the geometric $\mathbb{Z}$-submodule $I_K$ generated by \\
$\{ \gamma_{i_1,i_2,\dotsc,i_n}a_1^{i_1}a_2^{i_2}\dotsc a_n^{i_n}$ such that $\gamma_{i_1,i_2,\dotsc,i_n} = 1$ if $(i_1,i_2,\dotsc,i_n) \in K$ and $\gamma_{i_1,i_2,\dotsc,i_n} = \gcd(i_1,i_2,\dotsc,i_n)$ otherwise $\}$. 

In this $\mathbb{Z}$-submodule $I_K$, $\alpha_{w, I_K} = \alpha_w = \gamma_{w_1, w_2,\dotsc, w_n}$ where $w = \prod\limits_{j=1}^n a_j^{w_j}$. 

Then $\alpha_w | \gcd(w_1, w_2,\dotsc, w_n) | \sum\limits_{j=1}^n w_j(\sum\limits_{i = 1}^n \genby{a_j, a_i}v_i)) = \genby{w, v} | \genby{w, v}\alpha_v$, 
for every $v = \prod\limits_{i = 1}^n a_i^{v_i}\in A(n)$. 

\vspace{2mm}
By \cref{lm1}, this submodule $I_K$ is an ideal of the Lie algebra $\mathbb{Z}[A(n)]$. 

Because $K_0$ is finite, there is an infinite number of $n$ -tuples $(i_1, i_2,\dotsc, i_n) \not \in K_0$ such that $\gcd(i_1,i_2,\dotsc, i_n) > 1$. 

Therefore, we have an infinite sequence of $n$ tuples $\{\alpha_j\}_{i = 1}^{\infty}$ so that $\alpha_j = (j_1, j_2, \dotsc, j_n) \not \in K_0$ with $\gcd(j_1, j_2,\dotsc, j_n) > 1$. Then we can choose $K_j= K_0 \cup \{\alpha_t\}_{t = 1}^j$ so that we have an infinite sequence of distinct ideals $\{I_{K_j}\}$ that contains $K_0$ and therefore contains $X$. 
\end{proof}

By looking at its fundamental polygon, it is easy to see that a closed surface with genus g has the corresponding $A(n) = A(2g) = \genby{a_1, a_2,\dotsc, a_{2g}}$ with $n = 2g$ and $\genby{a_{2i-1}, a_{2i}} = 1$ and all other $\genby{a_i, a_j} = 0$ for $i \leq j$. Using this fact, we will prove the following proposition:

\begin{prop} Consider the algebra $\mathbb{Z}[A(n)]$ with $n = 2g$ that is associated with the closed surface $S$ with genus $g$. A geometric submodule $M$ of $\mathbb{Z}[A(n)]$ is an ideal if and only if for every $(i_1,\dotsc, i_{2g})$ and $(k_1,\dotsc,k_{2g}) \in \mathbb{Z}^{2g}$, $\alpha_M(k_1,k_2,\dotsc,k_{2g})$  divides $ \alpha_M(i_1,i_2,\dotsc,i_{2g}) \cdot \gcd\limits_{t = \overline{1, s}}(\gcd(k_{2t-1},k_{2t}) \cdot \gcd(i_{2t-1},i_{2t}))$.
\end{prop}

\begin{proof}
We prove the first direction of the theorem by assuming that $M$ is an ideal. By identifying $a_i$ with edges of the fundamental polygon of closed surface $S$, we have $\genby{a_{2i-1}, a_{2i}} = 1$ for all $i = \overline{1,n}$ and $\genby{a_i, a_j} = 0$ for other $i \leq j$.

Using \cref{lm1}, we have: \begin{equation} \label{eq:1}
\alpha_M(k_1,k_2,\dotsc,k_{2g}) | \alpha_M(i_1,i_2,\dotsc,i_{2g})\sum\limits_{t = 1}^g (i_{2t-1}k_{2t}-k_{2t-1}i_{2t}).
\end{equation} 
From equation $\hyperref[eq:1]{1}$, we have: $\alpha_M(1,0,\dotsc,0)$ divides $i_2 \alpha_M(i_1,i_2,\dotsc,i_{2g})$. 

Then $\alpha_M(1, 0, \dotsc, 0)$ divides $\alpha_M(0, 1,\dotsc, 0)$. Similarly, $\alpha_M(0,1,\dotsc,0)$ divides $\alpha_M(1,0,\dotsc,0)$ and therefore $\alpha_M(1, 0,\dotsc, 0) = \alpha_M(0, 1, \dotsc, 0)$. 

Then, again by equation $\hyperref[eq:1]{1}$, we got $\alpha_M(1, 0,\dotsc, 0) = \alpha_M(0, 1, \dotsc, 0)$ divides $i_1 \alpha_M(i_1,i_2,\dotsc,i_{2g})$. Therefore, $\alpha_M(0,1,\dotsc,0)$ divides $\gcd(i_1, i_2) \alpha_M(i_1,i_2,\dotsc,i_{2g})$\\

From equation $\hyperref[eq:1]{(1)}$, we also have $\alpha_M(k_1,k_2,\dotsc,k_{2g})$ divides $k_2\alpha_M(1,0,\dotsc,0)$ and $k_1\alpha_M(0,1,\dotsc,0)$. So $\alpha_M(k_1,k_2,\dotsc,k_{2g})$ divides $\gcd(k_1, k_2) \alpha_M(1,0,\dotsc,0)$.\\

Therefore,
\begin{equation*}
\alpha_M(k_1,k_2,\dotsc,k_{2g}) | \gcd(k_1, k_2)\alpha_M(1,0,\dotsc,0)| \gcd(k_1,k_2)  \gcd(i_1,i_2) \cdot \alpha_M(i_1,i_2,\dotsc,i_{2g}) 
\end{equation*}

Similarly, we can prove that, and so we finish the proof for the first direction.\\

The other direction that if $\alpha_M(k_1,k_2,\dotsc,k_{2g})$ divides \\ $\alpha_M(i_1,i_2,\dotsc,i_{2g}) \cdot \gcd\limits_{t = \overline{1, s}}(\gcd(k_{2t-1},k_{2t}) \cdot \gcd(i_{2t-1},i_{2t}))$, then $M$ is an ideal is obvious from \cref{lm1}.
\end{proof}

Similarly, using \cref{lm6}, we have a similar proposition for surface with boundary

\begin{prop} Consider the Lie algebra $\mathbb{Z}[A(n)]$ with $n = 2g+b-1$ that is associated with the surface $S$ which has genus $g$ and $b$ boundary components. A geometric submodule $M$ of $\mathbb{Z}[A(n)]$ is an ideal if and only if for every $(i_1,\dotsc, i_{n})$ and $(k_1,\dotsc,k_{n}) \in \mathbb{Z}^{n}$, $\alpha_M(k_1,k_2,\dotsc,k_{n})$  divides $ \alpha_M(i_1,i_2,\dotsc,i_{n}) \cdot \gcd\limits_{t = \overline{1, g}}(\gcd(k_{2t-1},k_{2t}) \cdot \gcd(i_{2t-1},i_{2t}))$.
\end{prop}

\section{Classification of symmetric ideals of $\mathbb{Q}$-module of freely homotopic classes on compact orientable surface} \label{st3}
\begin{defi}\label{def6}
For a surface $S$, consider the Lie algebra $\mathbb{Q}[A(n)]$, with the correponding abelianization of $\pi_1(S)$, $A(n)$, which is defined in \cref{def2}. Consider the surjective Lie algebra homomorphism $\mathbf{Ab}_{\mathbb{Q}}$. A ideal $I$ of $\mathbb{Q}[\hat{\pi_0}]$ is $\mathbf{symmetric}$ if it is the preimage of some ideal in $\mathbb{Q}[A(n)]$ under the map $\mathbf{Ab}_\mathbb{Q}$.
\end{defi}

Now we consider the surface $S$, and suppose the abelianization of its fundamental group, $A(n)$, is generated by $a_1, a_2, \dotsc, a_n$. We will find all ideals of the Lie algebra $\mathbb{Q}[A(n)]$, and, therefore, all symmetric ideals of the Goldman algebra $\mathbb{Q}[\hat{\pi_0}]$ on surface $S$

\begin{defi}	
Let $C_S = \{x \in A(n) |\ \genby{x, y} = 0 \ \forall y \in A(n) \} = \{x \in A(n)|\  \genby{x, a_i} = 0 \ \forall i = \overline{1,n} \}$ be the center of $\mathbb{Q}[A(n)]$.
\end{defi}
\begin{defi}
Let $A$ be the matrix with entries $A_{ji} = \genby{a_i, a_j}$ for ${i, j = \overline{1,n}}$. For $x = \displaystyle\prod_{i=1}^n a_i^{x_i} \in A_n$, let $X =(x_1, x_2,\dotsc, x_n) \in \mathbb{Z}^n \subset \mathbb{Q}^n$, and let 
\begin{equation*}
M(x) = A.X = (M_1(x),M_2(x),M_3(x),\dotsc,M_n(x)) \in \mathbb{Z}^n \subset \mathbb{Q}^n, 
\end{equation*}
where $M_j(x) = \displaystyle\sum_{i=1}^n\genby{a_i, a_j} x_i$

\vspace{2mm}
Then we have:
\begin{equation*}
\genby{x, y} = \displaystyle\sum_{i, j = 1}^n \genby{a_i, a_j}x_iy_j = \displaystyle\sum_{j=1}^n y_j(\displaystyle\sum_{i=1}^n\genby{a_i, a_j}x_i) = \displaystyle\sum_{j=1}^n y_j M_j(x) = Y . M(x)
\end{equation*}
, for $x = \displaystyle\prod_{i=1}^n a_i^{x_i}, y = \displaystyle\prod_{i=1}^n a_i^{y_i}$
\end{defi}
Here $Y . M(x)$ is normal dot product between $2$ vectors $Y$ and $M(x)$ of the vector space $\mathbb{Q}^n$.

\begin{lem}\label{lm2}
For $x, y \in A_n$, there is no $v \in A_n$ such that $\genby{x, v} \neq 0$ and $\genby{y, v} = 0$ if and only if $x^{k_1} = cy^{k_2}$ for some $c \in C_S, k_1 \neq 0$, and $k_2 \in \mathbb{Z}$.
\end{lem}
\begin{proof}
Firstly, if there is no $v \in A_n$ such that $\genby{x, v} \neq 0$ and $\genby{y, v} = 0$, then $\genby{x, v} = 0$ whenever $\genby{y, v} = 0$. But this means that for every $V \in \mathbb{Z}^n$ and $V . M(y) = 0$, $V . M(x) = 0$. Hence, the subspace $M(y)^{\bot}$ of $\mathbb{Q}^n \subset$ the subspace $M(x)^{\bot}$ of $\mathbb{Q}^n$. If $M(x)$ is non-zero vector in $\mathbb{Q}^n$, then $M(y)$ is also a non-zero vector, and both of the subspaces of $\mathbb{Q}^n$, $M(x)^{\bot}$ and $M(y)^{\bot}$, have the same $(n-1)$ dimension. Hence they must be equal and therefore also equal to the subspace $\{M(x), M(y)\}^{\bot}$. Then the set $\{M(x), M(y)\}$ must be linearly dependent in $\mathbb{Q}^n$. 

Hence, $M(x) = kM(y)$ for some $k \in \mathbb{Q}$, or $k_1M(x) = k_2M(y)$ for some $k_1 \neq 0, k_2 \in \mathbb{Z}$. Therefore, $Ak_1X = Ak_2Y$ or $(k_1X - k_2Y) \in KerA$, which means that $M(x^{k_1}y^{-k_2}) = 0 \in \mathbb{Q}^n$. Let $c = x^{k_1}y^{-k_2}$. Then $\genby{c, y} = Y.M(c) = y.0 = 0$ for every $y \in G$, and therefore $c \in C$, completing the proof of \cref{lm4}. (Since the other direction is trivial)
\end{proof}

\begin{defi}
For a finite sequence of fixed distinct $k$ elements of $C_S$, $\{c_i\}_{i=1}^k$, and a finite sequence of non-zero rational numbers, $\{q_i\}_{i=1}^k$, we define $X_{\alpha}$ as $\{ \sum_{i=1}^k q_ic_ix, x \in A_n \setminus C_S \}$, where $\alpha = \alpha(c_1, \dotsc, c_k, q_1, \dotsc, q_k)$ depends only on ${c_i}$ and ${q_i}$.
\end{defi}

\begin{lem}\label{lm3}
Let $\{q_i\}_{i = 1}^k$ be a fixed sequence of non-zero rational numbers, and $\{c_i\}_{i = 1}^k$ be a sequence of $k$ distinct elements in $C_S$. Denote $f(x) = \sum_{i = 1}^k q_ic_ix$, and so $X_{\alpha} = X_{\alpha(c_1, \dotsc, c_k, q_1, \dotsc, q_k)}=\{f(x) | x \in A_n \setminus C_S \}$. Now suppose $I$ is an ideal of the Lie algebra $\mathbb{Q}[A(n)]$. Then if $I$ contains some element of the form $f(x) \in X_{\alpha}$ ($x \in A_n \setminus C_S$), then $X_{\alpha} \subset I$.

As a result, in the case for closed surface $S$, where $C_S = \{e\}$, if $I$ contains some $x \in A_n \setminus \{e\}$, then $A_n \setminus \{e\} \subset I$. (Simply because in this case $X_{\alpha} = A_n$ for all $\alpha$)
\end{lem}
\begin{proof}

\begin{sublem} \label{sl1}
If $f(x) \in I$, and $\genby{x, y} \neq 0$, then $f(y) \in I$
\end{sublem}
\begin{proof}
Suppose that $f(x) \in I$.
\begin{equation*}
[f(x), x^{-1}y] = \sum\limits_{i = 1}^k q_i \genby{c_ix, x^{-1}y} c_iy = \genby{x, y} (\sum\limits_{i = 1}^k q_ic_iy) = \genby{x, y} f(y) \in I
\end{equation*}
because $\genby{c_ix, x^{-1}y} = \genby{c_i, x^{-1}y} + \genby{x, x^{-1}y} = 0 + \genby{x, x^{-1}} + \genby{x, y} = \genby{x, y} \neq 0$. 

Hence $f(y)$ is also in $I$. 
\end{proof}

\begin{sublem}\label{sl2}
Suppose that $f(x) \in I$ with $M_i(x) \neq 0$ for some $i$. Then for every $z \in A_n$ with $M_i(z) \neq 0$, $f(z)$ also $\in I$.
\end{sublem}
\begin{proof}
Take $y_0 \in A_n$ with $Y_0 = (0,\dotsc,0, 1,0,\dotsc,0)$, with $0$ is in $i$th position.

We have $\genby{x, y_0} = Y.M(x) = M_i(x) \neq 0$. By \cref{sl1}, $f(y_0)$ is also $\in I$. Now, for any $z \in A_n$ with $M_i(z) \neq 0$, because $y_0 \in I$, and $\genby{y_0, z} = -\genby{z, y_0} = -Y_0.M(z) = -M_i(z) \neq 0$, $f(z)$ also $\in I$ by \cref{sl1}.
\end{proof}

Now come back to the proof of \cref{lm3}. By the lemma's assumption, there is some $x_0 \not\in C_S$: $f(x_0) \in I$. Since $x_0 \not\in C_S$, there exists $i$ such that $M_i(x_0) \neq 0$. 

Now suppose that $x$ is any element of $A_n$ that is not in $C_S$. Then $M_j(x) \neq 0$ for some $j$. We will prove that $f(x) \in I$

Assume $r_i$ is ith row of matrix $A$. Then $M_i(z) = r_i.Z$. Since $M_i(x), M_j(y) \neq 0$, $r_i$ and $r_j$ must be different from zero vectors. Therefore, $r_i^{\bot}$ and $r_j^{\bot}$ are two subspace of $\mathbb{Q}^n$ with dimension at most $n-1$. Hence there is some $Z \in \mathbb{Q}^n$ that is outside the union of the two subspaces. Choose $z = \prod\limits_{i = 1}^n a_i^{z_i} \in A_n$, with $Z = (z_1, z_2, \dotsc, z_n)$. So we have $M_i(z) = r_i.Z$, and $M_j(z) = r_j.Z$ are both non-zero vectors of $\mathbb{Q}^n$. 

Since $f(x_0) \in I$, and $M_i(x_0), M_i(z) \neq 0$, we have $f(z) \in I$ by \cref{sl2}. Since $f(z) \in I$, and $M_j(z), M_j(x) \neq 0$, we get $f(x) \in I$ by \cref{sl2} again. Therefore $f(x) \in I$ for every $x \in A_n \setminus C_S$, and so $X_{\alpha} \subset I$.
\end{proof}

\begin{theorem}
Every ideal $I$ of the Lie algebra $\mathbb{Q}[A(n)]$ on closed (compact) surface $S$ is $\mathbb{Q}[\{e\}]$ or $\mathbb{Q}[A_n \setminus \{e\}]$ or $\mathbb{Q}[A_n]$. Here $\mathbb{Q}[X]$ is the $\mathbb{Q}$-submodule of $\mathbb{Q}[A(n)]$ generated by $X$
\end{theorem}
\begin{proof}
Suppose $I \neq \mathbb{Q}[\{e\}]$. Among all elements $\in I \setminus \mathbb{Q}[\{e\}] \neq \emptyset$, choose $\gamma_0$ to be one of those elements that has smallest number of terms in $A_n$. 

Suppose that $\gamma_0 = q_1u_1 + q_2u_2 + \dotsc + q_ku_k, q_i \neq 0 \in \mathbb{Q}$, and $u_i \in A_n$ are distinct.

If there is some $u_i = e$, and, WLOG, that is $u_1 = e$. Then $k \geq 2$, and $q_2 \neq e$. As a result, $q_2 \not \in C_S$, and there exists $y \in A_n$ such that $\genby{q_2, y} \neq 0$. Because $\genby{q_2, q_2^{-1}} = 0$, $y \neq q_2$, and so $[\gamma_0, y] = \sum\limits_{i = 2}^k q_i \genby{u_i, y} u_iy \not \in \mathbb{Q}[\{e\}]$, and is in the ideal $I$. However, it has less than $\gamma_0$ at least $1$ term in $A_n$. Contraddiction. Therefore, $u_i \neq e$ for all $i = \overline{1, k}$

Now suppose there is some $v \in G$ such that $\genby{u_i, v} \neq 0$ and $\genby{u_j, v} = 0$. Then the element $\gamma_1 = [\sum\limits_{1 \leq t \leq k} q_tu_t, v] = \sum\limits_{1 \leq t \neq j \leq k}q_t\genby{u_t, v} u_tv \in I \setminus \mathbb{Q}[\{e\}]$ has at most $k-1$ terms in $A_n$. Contradiction. Therefore, there is no $v \in G$ such that $\genby{u_i, v} \neq 0$ and $\genby{u_j, v} = 0$. Hence, for any $1 \leq i \neq j \leq k$, $\genby{u_i, v} = 0$ whenever $\genby{u_j, v} = 0$. From here, we also have $\genby{u_i, u_j} = 0\ \forall 1 \leq i, j \leq k$

Now consider any $1 \leq i_0 \neq j_0 \leq k$. Because $u_{i_0} \not\in C_S$, there exists some $t \in G$ such that $\genby{u_{i_0}, t} \neq 0$. Then, by the above statement, $\genby{u_l, t} \neq 0\ \forall 1 \leq l \leq k$. 

As a result, $\beta_0 = [\sum_{i=1}^k q_iu_i, t] = \sum_{i=1}^k (\genby{u_i, t} q_i).u_it$ have exactly $k$ terms in $A_n$, and, therefore, also has the smallest number of terms in $A_n$. By a similar argument as before, we get $\genby{u_{i_0}t, v} = 0$ whenever $\genby{u_{j_0}t, v} = 0$

By \cref{lm2}, $(u_{i_0}t)^{k_1} = (u_{j_0}t)^{k_2}$ for some $k_1 \neq 0, k_2 \in \mathbb{Z}$. Hence $t^{k_1 - k_2} = u_{j_0}^{k_2}u_{i_0}^{-k1}$. As a result, $\genby{u_{i_0}, t^{k_1 - k_2}} = \genby{u_{i_0}, u_{j_0}^{k_2}u_{i_0}^{-k_1}} = k_2\genby{u_i, u_j} - k_1\genby{u_i, u_i} = 0 - 0 = 0$. 

If $k_1 \neq k_2$, then $\genby{u_{i_0}, t} = \frac{1}{k_1 - k_2}\genby{u_i, t^{k_1 - k_2}} = 0$. Contradiction. Therefore, $k_1 = k_2 \neq 0$. Hence, $u_{i_0} = u_{j_0}$. Therefore, $u_i = u_j$ with every $i \neq j$. As a result, $k$ has to be $1$.

As a result, $u_1 = \frac{1}{q_1}\gamma_0 \in I$ for some $u_1 \in A_n \neq e$. By \cref{lm3}, we have $A_n \setminus e \subset I$. Therefore, $I$ can only be $\mathbb{Q}[\{e\}]$, $\mathbb{Q}[A_n \setminus \{e\}]$ or $\mathbb{Q}[A_n]$.
\end{proof}

\begin{defi}
For $x \in \mathbb{Q}[A_n]$, suppose $x = \sum_{i = 1}^m q_iu_i$, where $u_i \in A_n$, and $q_i$ are non-zero rational numbers. 

We say that $u_i \sim u_j$ if $u_iu_j^{-1} \in C_S$. Then the set $\{u_1, u_2,\dotsc, u_m\}$ is partitioned into equivalent classes $R_i$ for $1 \leq i \leq p$.

Then $x = \sum_{i = 1}^p (\sum_{u_j \in R_i}q_ju_j) = \sum_{i = 1}^p (\sum_{u_j \in R_i} q_jc_ju_{h_i})$, where $u_{h_i}$ is a representative of the equivalent class $R_i$.

Note that each inner sum is either $\in \mathbb{Q}[C_S]$, the $\mathbb{Q}$-submodule of $\mathbb{Q}[A_n]$ generated by $C_S$, or equal to $f(u_{h_i}) \in X_{\alpha_i}$.

So $x$ can be rewritten as $\sum_{i = 1}^{p_1} f(v_i) + c$, for $f(v_i)$ in some $X_{\alpha_i}$, some $c \in \mathbb{Q}[C_S]$, and some $p_1 \leq p$. Moreover, by using the previous process, this representation of $x$, and we will call this the $\mathbf{standard}$ representation of $x \in \mathbb{Q}[A_n]$.
\end{defi}

\begin{lem}\label{lm4}
For any $y \in A_n$, and $f_{\alpha}(x)$ that is in some $X_{\alpha}$, we have $[f_{\alpha}(x), y] = \genby{x, y}f_{\alpha}(xy)$. As a result, $\mathbb{Q}[X_{\alpha}]$ is an ideal of $\mathbb{Q}[A(n)]$. We will call such ideal $\mathbb{Q}[X_{\alpha}]$ a \textbf{primitive} ideal
\end{lem}
\begin{proof}
\begin{equation*}
[f_{\alpha}(x), y] = [\sum_i q_ic_ix, y] = \sum_i q_i\genby{c_ix, y}c_ixy = \sum_i \genby{x, y}q_ic_ixy = \genby{x, y}f_{\alpha}(xy)
\end{equation*}
\end{proof}

\begin{lem}\label{lm5}
Suppose that an ideal $I$ of the Lie algebra $\mathbb{Q}[A(n)]$ has an element whose standard representation contains a term that is an element of some $X_{\alpha_0}$. Then $X_{\alpha_0} \subset I$. 
\end{lem}

\begin{proof}
Among all of elements of $I$ that contains a term $\in X_{\alpha_0}$ in its standard expression, choose $\gamma_0$ to be the element with the smallest number of terms in $A_n$. Note that the set of elements of $I$ containing a term $\in X_{\alpha_0}$ in its standard expression is non-empty because of the lemma's assumption. 

Suppose $\gamma_0$ has the standard representation $\gamma_0 = \sum\limits_{i = 0}^k f_{\alpha_i}(x_i) + c$, where $x_i$ and $x_ix_j^{-1} \in A_n \setminus C_S$ for $1 \leq i \neq j \leq k$, and $c \in C_S$. In this representation, $f_{\alpha_0}(x_0)$ is the term that belongs to $X_{\alpha_0}$. We will prove that $c = 0$ and $k = 0$. 

If $c \neq 0$, then take $y \in A_n$ such that $\genby{x_0, y} \neq 0$. Using \cref{lm4}, we get: 
\begin{multline*}
\gamma'_0 = \dfrac{1}{\genby{x_0, y}}[\gamma_0, y]  = \sum\limits_{i =0}^k \dfrac{1}{\genby{x_0, y}}[f_{\alpha_i}(x), y] \\
= \sum\limits_{i = 0}^k\dfrac{\genby{x_i, y}}{\genby{x_0, y}}f_{\alpha_i}(x_iy) = \sum\limits_{\substack{0 \leq i \leq k \\ \genby{x_i, y} \neq 0}} f_{\beta_i}(x_iy) \in I
\end{multline*}
for some $f_{\beta_i} \in X_{\beta_i}$.

Because for $1 \leq i \neq j \leq k$, $x_iy(x_jy)^{-1} = x_ix_j^{-1} \not \in C_S$, $\sum\limits_{\substack{0 \leq i \leq k \\ \genby{x_i, y} \neq 0}} f_{\beta_i}(x_iy)$ is the standard representation of $\gamma'_0$. 

Also, in this presentation, the first term is $f_{\beta_0}(x_0y) = f_{\alpha_0}(x_0y) \in X_{\alpha_0}$ (Note that $\genby{x_0, y} = \genby{x_0y, y} \neq 0$, so $x_0y \not \in C_S$). However, $\gamma'_0$ has at most $k + 1$ terms in the standard representation, and this is impossible because of the choice of $\gamma_0$, which has $k + 2$ terms. So $c = 0$.

\vspace{2mm}
Now suppose, by contradiction, that $k > 0$.

Now for each $1 \leq i_0 \leq k$, if there is some $y \in G$ such that $\genby{x_0, y} \neq 0$ and $\genby{x_{i_0}, y} = 0$, then:
\begin{equation*}
\begin{split}
\gamma_1 &= \dfrac{1}{\genby{x_0, y}}[\gamma_0, y]  = \sum\limits_{i = 0}^k \dfrac{1}{\genby{x_0, y}}[f_{\alpha_i}(x), y] \\
& = \sum\limits_{i = 0}^k \dfrac{\genby{x_i, y}}{\genby{x_0, y}}f_{\alpha_i}(x_iy) = \sum\limits_{\substack{0 \leq i \neq i_0 \leq k \\ \genby{x_i, y} \neq 0}} f_{\beta_i}(x_iy) \in I
\end{split}
\end{equation*}
for some $f_{\beta_i} \in X_{\beta_i}$.

Again, $\sum\limits_{\substack{0 \leq i \neq i_0 \leq k \\ \genby{x_i, y} \neq 0}} f_{\beta_i}(x_iy)$ is a standard representation of $\gamma_1$ that contains $f_{\beta_0}(x_0  y) = f_{\alpha_0}(x_0y) \in X_{\alpha_0}$. However, $\gamma_1$ has at most $k$ terms, while $\gamma_0$ has $k + 1$ terms in their standard presentation, and this is impossible.

As a result, $\genby{x_0, y} = 0$ whenever $\genby{x_i, y} = 0$ for all $1 \leq i \leq k$.

Therefore, for any $1 \leq i \leq k$, $\genby{x_0, x_i} = 0$. Moreover, by \cref{lm2}, we get $x_0 ^{k_1^{(i)}} = c_ix_i^{k_2^{(i)}}$ for some $c_i \in C_S$ and $k_1^{(i)} \neq 0, k_2^{(i)} \in \mathbb{Z}$.\\

Since $x_0 \not\in C_S$, there exists some $y_0 \in G$ such that $\genby{x_0, y_0} \neq 0$. Consider \begin{equation*}
\begin{split}
\gamma_2 &= \dfrac{1}{\genby{x_0, y_0}}[\gamma_0, y_0]  = \sum\limits_{i = 0}^k \dfrac{1}{\genby{x_0, y_0}}[f_{\alpha_i}(x_i), y_0] \\
& = \sum\limits_{i = 0}^k \dfrac{\genby{x_i, y_0}}{\genby{x_0, y_0}}f_{\alpha_i}(x_iy_0) = \sum\limits_{\substack{0 \leq i \leq k \\ \genby{x_i, y_0} \neq 0}} f_{\beta'_i}(x_iy_0) \in I
\end{split}
\end{equation*}
for some $f_{\beta'_i} \in X_{\beta_i}$.

Again, $\sum\limits_{\substack{0 \leq i \leq k \\ \genby{x_i, y_0} \neq 0}} f_{\beta'_i}(x_iy_0)$ is the standard representation of $\gamma_2$ with at most $k + 1$ terms and contains the term $f_{\alpha_0}(x_0y_0) \in X_{\alpha_0}$. Therefore, $\gamma_2$ must also have the smallest number of terms in its standard representation. 

Using the same argument, which we used for $\gamma_0$, for $\gamma_2$,  we have for $1 \leq i \leq k$, $(x_0y_0)^{k_3^{(i)}} = d_i(x_iy_0)^{k_4^{(i)}}$ for some $d_i \in C_S$, and $k_3^{(i)} \neq 0, k_4^{(i)} \in \mathbb{Z}$. Therefore, $y_0^{k_3^{(i)}-k_4^{(i)}} = x_i^{k_4^{(i)}}x_0^{-k_3^{(i)}}d_i$. Because $\genby{x_i, x_0} = 0$ and $c \in C_S$, $\genby{x_0, y_0}^{k_3^{(i)}-k_4^{(i)}} = 0$.

If $k_3^{(i)} \neq k_4^{(i)}$, we would have $\genby{x_0, y_0} = \frac{1}{k_3^{(i)}-k_4^{(i)} } \genby{x_0, y}^{k_3^{(i)}-k_4^{(i)} } = 0$. Contradiction. 

So $k_3^{(i)}= k_4^{(i)} \neq 0$, and, therefore, for $1 \leq i \leq k$, $x_ix_0^{-1} \in C$. Contradiction because $x_ix_j^{-1} \not \in C_S$ for $1 \leq i \neq j \leq k$. 

\vspace{2mm}
As a result, $f_{\alpha_0}(x_0) = \gamma_0 \in I$, so $I$ contains some element $f_{\alpha_0}(x_0)$ of $X_{\alpha_0}$, and, by \cref{lm3}, $X_{\alpha_0} \subset I$.
\end{proof}

\begin{theorem}(Theorem for classification of ideals) 
	
All ideals $I$ of the Lie algebra $\mathbb{Q}[A_n]$ is the sum, as vector space, of $\{\mathbb{Q}[X_{\alpha_t}]\}_{t \in T}$ and $C_1$, where $C_1$ is a subspace of $\mathbb{Q}[C_S]$ and can be empty, and some index set $T$. Here $\mathbb{Q}[T]$ is the $\mathbb{Q}$-submodule of $\mathbb{Q}[A(n)]$ generated by $T$.
\end{theorem}

\begin{proof}
Any subspace $C_1$ of $\mathbb{Q}[C_S]$ is obviously an ideal. Therefore, if a $\mathbb{Q}$-submodule (or subspace) $I$ is the sum of some of primitive ideals (defined in \cref{lm4}), and $C_1$, then $I$ must also be an ideal of $\mathbb{Q}[A_n]$.

Now we will prove that if $I$ is an ideal, then $I$ must be the sum of primitive ideals $\mathbb{Q}[X_{\alpha}]$ and some subspace $C_1$ of $\mathbb{Q}[C_S]$. 

Let $\Gamma$ be the set of all $\alpha$ such that in the standard representation of some element of $I$, there is a term $f_{\alpha}(x)$ which is an element of $X_{\alpha}$. Let $C$ be the set of all terms that are in $\mathbb{Q}[C_S]$ and appears in the standard representation of some elements of $I$. Clearly, $I$ is the subset of $\displaystyle\bigoplus_{\alpha \in \Gamma} \mathbb{Q}[X_{\alpha}]\ \bigoplus\  \mathbb{Q}[C_1]$. 

\vspace{2mm}
Now, by \cref{lm5}, $X_{\alpha}$ is a subset of $I$  for every $\alpha \in \Gamma$. Therefore, $\displaystyle\bigoplus_{\alpha \in \Gamma} \mathbb{Q}[X_{\alpha}]$ is also a subset of $I$. 

Now, for each element $c$ of $C$, there exists $x \in I$, with the standard representation $x = c + d$, where $c \in C \subset \mathbb{Q}[C_S]$, and $d = \sum\limits_{i = 1}^k f_{\alpha_i}(x_i)$. Each $\alpha_i \in \Gamma$ because of the definition of $\Gamma$, so $d \in \displaystyle\bigoplus_{\alpha \in \Gamma} \mathbb{Q}[X_{\alpha}]$. Therefore, $d \in I$, and hence, $c \in I$. Therefore, $\mathbb{Q}[C] \subset I$. Hence, $\displaystyle\bigoplus_{\alpha \in \Gamma} \mathbb{Q}[X_{\alpha}]\ \bigoplus\  \mathbb{Q}[C_1] \subset I$

As a result, $I = \displaystyle\bigoplus_{\alpha \in \Gamma} \mathbb{Q}[X_{\alpha}]\ \bigoplus\  \mathbb{Q}[C_1]$. 
\end{proof}

Now the following lemma will help us compute the center $C_S$ explicitly.

\begin{lem}\label{lm6}
For the surface $S$ with boundary with genus $g$ and $b$ boundary components, there exists generators $a_1, a_2,..., a_n$ with $n = 2g + b - 1$ such that the abelianization of the fundamental group of $S$, $A_n = \genby{a_1, a_2,..., a_n | a_i a_j = a_j a_i}$ and $\genby{a_{2i-1}, a_{2i}} = 1 \ \forall i=\overline{1,g}$ and $\genby{a_i, a_j} = 0\ \forall i \leq j$.
\end{lem}
\begin{proof}
We consider a $4n$-gon with $2n$ labels $a_1, a_2,..., a_n, A_1, A_2,..., A_n$ such that only odd edges (the first, the third, and so on) have labels. Then we mark the $2n$ edges with $2n$ labels $a_1, a_2, A_1, A_2,...., a_{2i-1}, a_{2i}, A_{2i-1}, A_{2i},..., a_{2g-1}, a_{2g}, A_{2g-1}, A_{2g},$
and then $A_{2g+1}, a_{2g+1},..., A_{2g+b-1}, a_{2g+b-1}$ in this order. The following picture describe the above process for surface with genus 2 and with 1 boundary component:

\begin{center}
\includegraphics[scale=0.5]{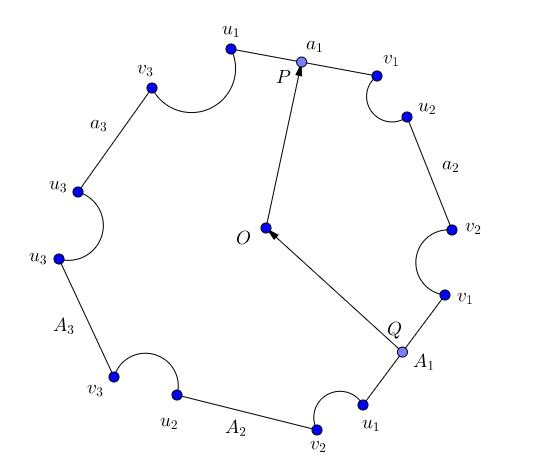}
\end{center}

We then glue $a_i$ and $A_i$ for $i=\overline{1,g}$ without creating Mobius band to obtain a surface. Denote $u_i$, and $v_i$ be the vertices of the side $a_i$ and $A_i$ for $i=\overline{1,g}$. We define that 2 vertices $a$ and $b$ are equivalent: $a \sim b$ if $a$ and $b$ are in the same small arcs, which are even edges of the $4n$-gon that make up the boundaries. Then we will have cycles of equivalent vertices. The number of cycles will be the number of boundaries. In fact, we have the cycle $u_1 \sim v_2 \sim v_1 \sim u_2 \sim u_3 \sim v_4 \sim v_3 \sim u_4 \sim ... \sim u_{2g-1} \sim v_{2g} \sim v_{2g-1} \sim u_{2g} \sim v_{2g+1} \sim v_{2g+2} \sim ... \sim v_{2g+b-1}$ and $b-1$ other cycles of 1 vertex $\{u_{2g+i}\}_{i=1}^{b-1}$.

Therefore the obtained surface is an oriented surface with $b$ boundary component and with genus $(n + 1 - b)/2 = g$. So this surface is homeomorphic to $S$ by the classfication theorem for surface (with boundary). So we can identify $S$ with this glued surface.

Now take $O$ inside the polygon. $P = Q$ are inside edges labeled by $a_i$, $A_i$ respectively. Let $\alpha_i$ be the curve $O \rightarrow P = Q \rightarrow O$. Denote $a_i$ as the homotopy classes of curve base at $O$ which has the curve $\alpha_i$ as an representation. Then the fundamental group of the glued surface, and also, $S$, will be the free group generated by $\{a_i\}_{i=1}^n$. Moreover, based on our labeling of the $4n$-gon, we can choose the curves $\{\alpha_i\}_{i = 1}^n$, which are freely homotopic to curves that are represented by $a_i$ in $\pi_1(S)$, so that $\alpha_i$ for $i = \overline{2g+1, 2g+b-1}$ interesects no other curve, and $\alpha_{2i-1}$ only intersects $\alpha_{2i}$ for $i = \overline{1,g}$ and no other curve. Hence, we will have the abelianization of $\pi_1(S)$, $A_n = \genby{a_1, a_2,..., a_n | a_i a_j = a_j a_i}$ such that $\genby{a_{2i-1}, a_{2i}} = 1 \ \forall i=\overline{1,g}$ and $\genby{a_i, a_j} = 0\ \forall i \leq j$ because of defintion of $\genby{a_i, a_j}$
\end{proof}

\begin{lem}
(Lemma for computing $C_S$) 
\begin{itemize}[leftmargin=0.1 in]
\item For closed surface $S$, $C_S = \{e\}$.

\item By \cref{lm6}, for surface $S$ with boundary, there exists generators $a_1, a_2,..., a_n$ with $n = 2g + b - 1$ such that the abelianization of the fundamental group of $S$, $A_n = \genby{a_1, a_2,..., a_n | a_i a_j = a_j a_i}$ and $\genby{a_{2i-1}, a_{2i}} = 1 \ \forall i=\overline{1,g}$ and other $\genby{a_i, a_j} = 0\ \forall i \neq j$. Then $C_S$ is the subgroup generated by $\{a_i\}_{i = 2g + 1}^n$.
\end{itemize}
\end{lem}
\begin{proof}
For the closed surface $S$, and let $a_j$ for $j \in \overline{1, n}$ with $n= 2g$ be generators of $A_n$ that corresponds to the $i^{th}$ sides of the fundamental polygon of $S$. Then since only $(2i-1)^{th}$ and $(2i)^{th}$ sides of the fundamental polygon intersects, we have $\genby{a_{2i-1}, a_{2i}} = 1$, and other $\genby{a_i, a_j} = 0$.

Now suppose $x = \prod\limits_{j=1}^na_j^{k_j}$ for some integers $k_j$. Then $0 = [x, a_{2i-1}] = k_{2i}$, and $0 = [x, a_{2i}] = k_{2i-1}$ for each $i$ from $1$ to $r$. Hence,  $C_S = \{e\}$.

\vspace{4mm}
For surface with genus $g$ and $b$ boundary components and with $n = 2g + b - 1$, by \cref{lm6}, $A_n = \genby{a_1, a_2,..., a_n | a_i a_j = a_j a_i}$ and $\genby{a_{2i-1}, a_{2i}} = 1 \ \forall i=\overline{1,g}$ and other $\genby{a_i, a_j} = 0\ \forall i \neq j$. Then $C_S$ is the subgroup generated by $\{a_i\}_{i = 2g + 1}^n$.

Now suppose $x = \prod\limits_{j=1}^n a_j^{k_j}$ is some element of $C_S$. Then $0 = [x, a_{2i}] = k_{2i-1}$, and $0 = [x, a_{2i-1}] = k_{2i}$ with $i \in \overline{1, n}$. Therefore, $x = \prod\limits_{j=2g+1}^n a_j^{k_j}$. Hence $C_S = \genby{a_{2g+1},...,a_n}$.
\end{proof}

Below is a corollary from the classification theorem of ideals and the lemma for computing the set $C_S$ for surface $S$ with more than $1$ boundary components.  

\begin{corollary}
There are infinitely many non-trivial symmetric ideals (defined in \cref{def6}) of the Goldman algebra $\mathbb{Q}[\hat{\pi_0}]$ on the surface with more than $1$ boundary components.
\end{corollary}

\section{An infinite descending chain of ideals in the Goldman algebra}
\begin{defi}\label{def11}
Consider the surface $S$ with at least 1 boundary component. Choose a fixed boundary. Then choose a curve (the red curve) going around the boundary exactly 1 time. Suppose we have a curve $\alpha$ that touched the fixed red curve only at point $P$.

\begin{center}
\includegraphics[scale=0.5]{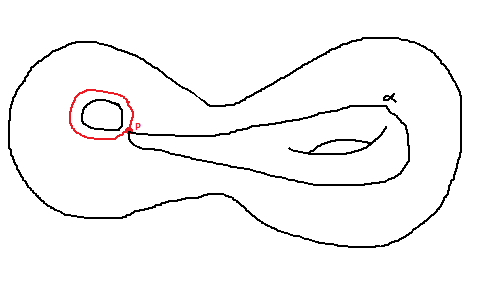}
\end{center}

We define $T_0^{+}$ operation be the operation that transforms the curve $\alpha$ into a new curve $\beta$ that goes around the red curve 1 time in counter-clockwise direction and then goes around the original curve $\alpha$. 

\begin{center}
\includegraphics[scale=0.5]{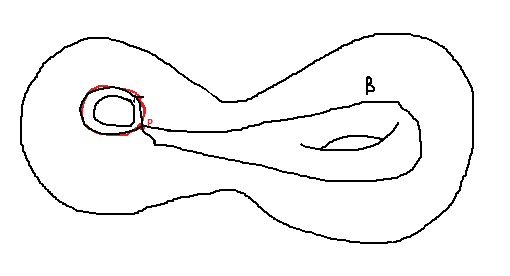}
\end{center}

Note that by another homotopy operation, $\beta$ can further be transformed into $\gamma$. 

\begin{center}
\includegraphics[scale=0.5]{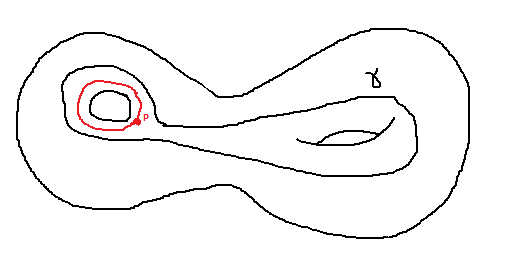}
\end{center}

Now we define the inverse operation of $T_0^{+}$. If we have a loop $\beta$ that goes around the red curve exactly one, and then goes around another $\alpha$ curve which only touches the red curve at $P$ then we define $T_0^{-}$ to be the operation that transforms $\beta$ into $\alpha$. Finally, we can define the operation $T_0$ to be either $T_0^{+}$ or $T_0^{-}$.
\end{defi}

\begin{lem}\label{lm8}
Suppose that $\alpha$ is transformed into $\beta$ by an operation $T_0$. Moreover, suppose $\zeta$ is another curve that doesn't intersect the red curve going around the boundary, then each term in the Goldman bracket $[\alpha, \zeta]$ can be transformed into a corresponding term in the bracket $[\beta, \zeta]$ by the operation $T_0$. (This correspondence is one-to-one).
\end{lem}
\begin{proof}
WLOG, assume that $\alpha$ is transformed into $\beta$ by $T_0^{+}$. Note that all intersection points of $\alpha$ and $\zeta$ are the same as those of $\beta$ and $\zeta$ since $\zeta$ doesn't intersect the red curve. Assume $Q$ is one of those intersection. Then clearly, the curve $\alpha_p\zeta$ can be transformed into $\beta_p\zeta$ by $T_0^{+}$. 

\begin{center}
\includegraphics[scale=0.6]{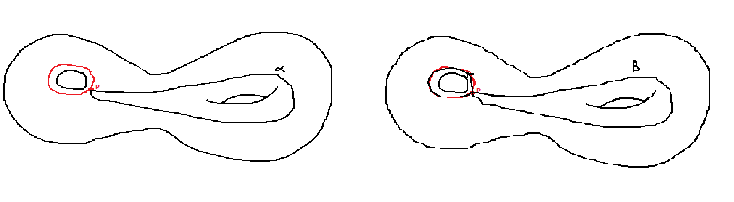}

\includegraphics[scale=0.6]{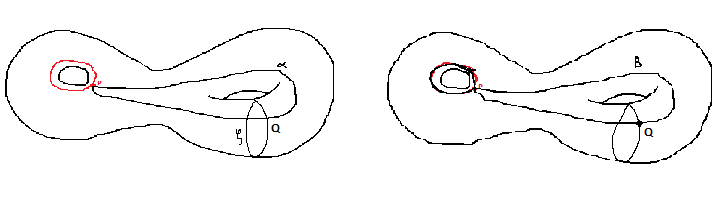}
\end{center}
\end{proof}

\begin{defi}
Now suppose we have a new red curve going around the boundary $2^1 = 2$ times. We define $T_1$ to be an operation that is similar to $T_0$, and the only difference is the new red curve.

\begin{center}
	
\includegraphics[scale=0.5]{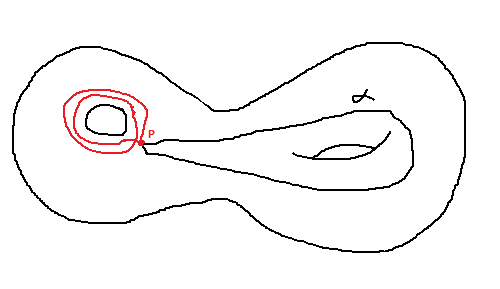}

\includegraphics[scale=0.5]{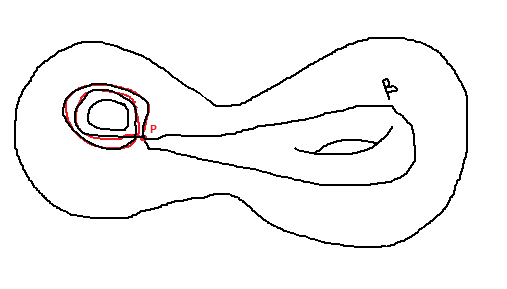}

\includegraphics[scale=0.5]{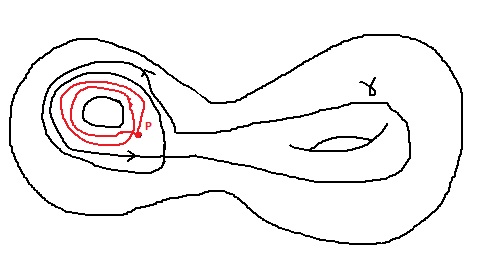}
	
\end{center}

Now instead of letting the red curve go around the boundary $1$ or $2$ times, we can let the red curve go around the boundary $2^n$ times for any non-negative integer $n$. Then we can define $T_n$ in a similar way that we defined $T_0$ and $T_1$.
\end{defi}

Then we have another version of \cref{lm8}

\begin{lem}\label{lm9}
Suppose that $\alpha$ can be transformed into $\beta$ by an operation $T_n$. Moreover, suppose $\zeta$ is another curve that doesn't intersect the red curve going around the boundary, then each term in the Goldman bracket $[\alpha, \zeta]$ can be transformed into a corresponding term in the bracket $[\beta, \zeta]$ by the operation $T_0$. (This correspondence is one-to-one).
\end{lem}

\begin{defi}
Two curves on $S$, $\alpha$ and $\beta$, are equivalent under equivalence relation $\sim_n$ if $\alpha$ can be transformed into $\beta$ by operations $T_m$ for some $m \geq n \in \mathbb{Z}$ or by homotopy operations. We can then extend this equivalence relation $\sim_n$ for linear combinations of curves instead of just single curves.
\end{defi}

\begin{lem} \label{lm10}
If two curves $\alpha \sim_n \beta$, then $[\alpha, \zeta] \sim_n [\beta, \zeta]$ 
\end{lem}
\begin{proof}
The proof follows from \cref{lm9}, and the fact that for any $\zeta$ there always exists a curve $\zeta'$, which is homotopic to $\zeta$ and doesn't intersect the curve that going around the boundary a certain number of times.
\end{proof}

\begin{defi}\label{def14}
We define $X_n$ as the set of all equivalence classes of curves obtained from the equivalence relation $\sim_n$. Now let $S_n$ be the set of all linear combinations of elements of $X_n$ with coefficients in $\mathbb{Q}$. $S_n$ is a $\mathbb{Q}$-module. 

Let the $\mathbb{Q}$-module homomorphism $\delta_n$ from the Goldman algebra to $S_n$ be the map that maps each freely homotopy class of curves to the equivalence classes in $X_n$ of one of the representation curve of that freely homotopy class. This map is well-defined since any two freely homotopic curves are equivalent under relation $\sim_n$.

By \cref{lm10}, we can define a Lie bracket on $S_n$: $[\overline{\alpha}, \overline{\beta}] = \overline{[\alpha, \beta]}$, where $\overline{\gamma}$ is the $\sim_n$ equivalence classes of $\gamma$. This definition doesn't depend on the representative curves $\alpha$ or $\beta$. 
\end{defi}

\begin{lem}\label{lm11}
By \cref{def14}, $S_n$ is a Lie algebra, and $\delta_n$ is obviously a Lie algebra homomorphism. Let $I_n$ be the kernel of the map $\delta_n$. Then we have a descending chain of ideals $I_n$ because any two curves that are $\sim_{n+1}$-equivalent are also $\sim_{n}$-equivalent.
\end{lem}

\begin{lem}\label{lm12}
Let $\pi_1(S, P)$ be the fundamental group of $S$ with base point $P$, and $c$ be an element of $\pi_1(S, P)$ that is the homotopy class of curves containing the chosen boundary (in counter-clockwise direction) (the red curve defined in \cref{def11}). 

Then any closed curve that is in the homotopy class of curves $cxc^{-1}x^{-1} \in \pi_1(S, P)$ is $\sim_0$ to the trivial loop. 

Moreover, for any $x_i \in \pi_1(S, P)$, and integers $m_i \geq n$, the curve that is in the homotopy class 
\begin{equation*}
\gamma = c^{2^{m_1}} x_1 c^{-2^{m_1}} x_1^{-1} c^{2^{m_2}} x_2 c^{-2^{m_2}}x_2^{-1}\dotsc c^{2^{m_k}} x_k c^{-2^{m_k}} x_k^{-1} \in \pi_1(S, P)
\end{equation*}
is $\sim_n$ to the trivial loop. 

Therefore, if 
\begin{multline*}
C_n = \{gc^{2^{m_1}} x_1 c^{2^{-m_1}} x_1^{-1} c^{2^{m_2}} x_2 c^{2^{-m_2}}x_2^{-1}\dotsc c^{2^{m_k}} x_k c^{2^{-m_k}} x_k^{-1}g^{-1} \\ \text{ such that } x_i, g \in \pi_1(S, P), m_i \geq n \}
\end{multline*}
then $[x] - [y] \in I_n$ with every $x, y$ in the normal subgroup $C_n$ of $\pi_1(S, P)$, where $[x]$ is the freely homotopic class of curves so that one of the curve in this class is in the homotopy class $x \in \pi_1(S, P)$. As a result, $I_n$ is not one of the ideals that we considered in \cref{st3}.
\end{lem}

\begin{proof}
Consider the following curve represented by $cxc^{-1}x^{-1}$ in the fundamental group.

\begin{center}
\includegraphics[scale=0.5]{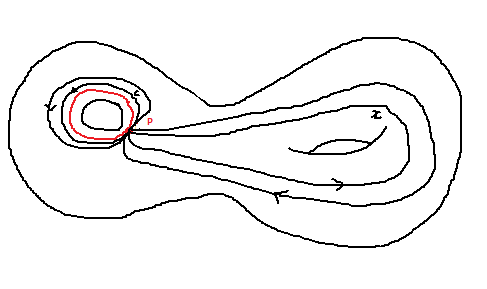}
\end{center}

This curve is homotopic to the curve $\alpha_0$

\begin{center}
\includegraphics[scale=0.5]{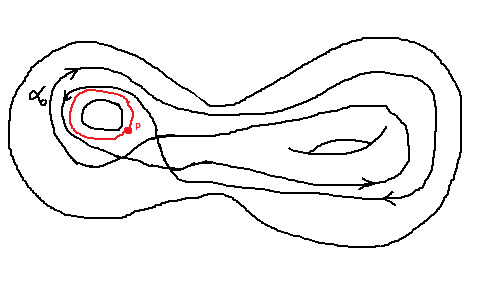}
\end{center}

We can perform the following operations on $\alpha_0$: $T_0$ operation, homotopy operation, then another $T_0$ operation, and finally a homotopy operation to transform $\alpha_0$ into the trivial loop.

\begin{center}
\includegraphics[scale=0.5]{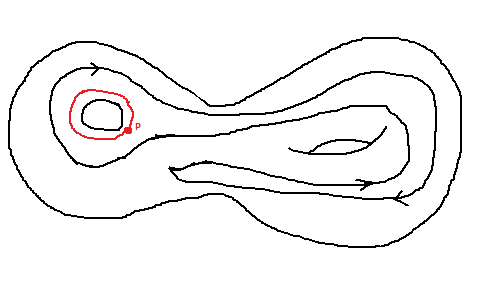}

\includegraphics[scale=0.5]{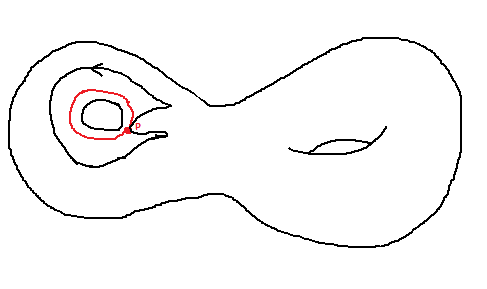}

\includegraphics[scale=0.5]{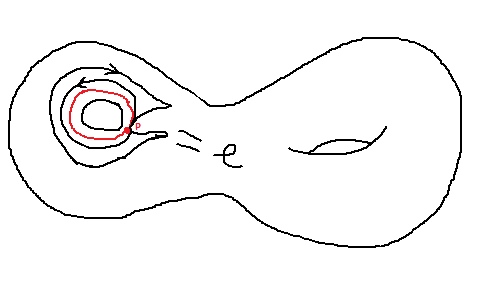}
\end{center}

Similarly, we can prove the more general statement for the equivalence relation $\sim_n$ instead of $\sim_0$. 
\end{proof}

\begin{lem}\label{lm13}
The freely homotopic class $\gamma$ containing the curve going around the boundary $2^n$ times belongs to $I_n$ and doesn't belong to $I_{n+1}$ for $n \in \mathbb{Z} \geq 0$. Therefore, $I_n$ is a strictly descending chain of ideals.
\end{lem}

\begin{proof}
By a single $T_n$ operation, we can transform any curve in the freely homotopic class $\gamma$ into the trivial loop, so $\gamma \in I_n$. Now we will prove that $\gamma \not \in I_{n + 1}$
	
Attaching a disk to the surface $S$ by gluing the disk's boundary with the curve that goes around the chosen boundary (of the surface $S$) $2^{n+1}$ times to obtain a new topological space $S'$. If we can transform one curve into another curve on the original surface $S$ by an operation $T_m$ for $m \geq n+1$, then we can also transform one curve into the other by a homotopy in $S'$ by moving the curve through the disk attached instead. 

Let $c \in \pi_1(S)$ be the homotopy class of curves containing the boundary that we have chosen. Then, by Van Kampen's theorem, the fundamental group of the new space is $\pi_1(S)/N$, where $N$ is the normal subgroup generated by $c^{2^n}$. If $\gamma \in I_{n+1}$, then, by our previous argument, $c^{2^n}$ must be homotopic to the trivial loop $S'$. Hence, $c^{2^n} \in N$ or $c^{2^n} = gc^{2^{n+1}}g^{-1}$, for some $g \in \pi_1(S)$. Contradiction. Therefore, $\gamma \not \in I_{n + 1}$
\end{proof}

\begin{theorem}\label{thm5}
For compact surface $S$ with more than one boundary component, we can choose our fixed boundary so that the homotopy class of curves containing this boundary is a generator of the fundamental group of $S$, which is a free group. In this case, the intersection of ideals $\bigcap\limits_{n \geq 0} I_n$ is $\{0_{\mathbb{Q}[\hat{\pi_0}]}\}$
\end{theorem}

\begin{proof}
Suppose by contradiction that there is some non-zero element in $\bigcap\limits_{n \geq 0} I_n$. Then there will be $2$ curves $\alpha$ and $\beta$ with different freely homotopic classes that are $\sim_n$ equivalent for each $n \in \mathbb{Z} \geq 0$. Suppose $\alpha$ and $\beta$ are in homotopy classes $a$ and $b$, which are elements of the fundamental group based at $P$, $\pi_1(S) = \pi_1(S, P)$. Then choose $n$ big enough such that the sum of absolute value of all exponents $k$ of subwords of the form $c^k$ in both $a$ and $b$ less than $2^{n-1}$. (Note that elements in $\pi_1(S)$ are reduced words).

Again, we attach to $S$ the disk whose boundary is the curve going around the boundary of the original surface $2^n$ times. Then $\alpha$ must be homotopic to $\beta$ in the new topological space $S'$.

For every reduced word in $\pi_1(S)$, $\dotsc c^{m_1}\dotsc c^{m_2}\dotsc \in \pi_1(S)$, we consider the corresponding word $\dotsc c^{m_1\mods 2^n}\dotsc c^{m_2\mods 2^n}\dotsc$, where we take $\mods 2^n$ of the exponents of subwords of the form of $c^i$. Let $G_n$ be the group of all of these corresponding words. We then have a natural surjective group homomorphism (or projection) $p$ from $\pi_1(S)$ to $G_n$. Every element in normal subgroup $N$ generated by $c^{2^n}$ is mapped to $\widetilde{e}$, unit element in $G_n$. As a result, there exist a group homomorphism $f$ from $\pi_1(S)/N$ to $G_n$ so that the following diagram commutes:

\[
\begin{tikzcd}
& \pi_1(S)/N \arrow{dr}{f} \\
\pi_1(S) \arrow{ur}{p_1} \arrow{rr}{p} && G_n
\end{tikzcd}
\]

Let $p_1(x) = \overline{x}$, and $f(\overline{x}) = \widetilde{x} = p(x)$.
Because $\alpha$ and $\beta$ are homotopic in the new space $S'$, $\overline{a}$ and $\overline{b}$ are conjugate in $\pi_1(S)/N$. Therefore, $\overline{agbg^{-1}} = \overline{e}$ for some $g \in \pi_1(S)$. 

Now we can assume that $g$ and therefore $g^{-1}$ only have subword $c^m$ in the $g$ with $m \in \mathbb{Z}$, and $|m| \leq 2^{n-1}$ because we can replace $m$ by $m'= m \mod 2^n$, with $0 \leq m' \leq 2^n-1$ and if $m'$ is still bigger than $2^{n-1}$, we then can replace $m$ by $m'' = m' - 2^n$. With these replacements, the identity $\overline{agbg^{-1}} = \overline{e} \in \pi_1(S)/N$ stays the same.

\vspace{2mm}
Now project two sides of the identity $\overline{agbg^{-1}} = \overline{e} \in \pi_1(S)/N$ onto $G_n$ by the map $f$, we got: $\widetilde{agbg^{-1}} = \widetilde{e}$. Note that $agbg^{-1}$ cannot be $e$ because the free homotopy classes of $\alpha$ and $\beta$ are different. 

Therefore, $agbg^{-1} = y$ so that $y = x_1c^{2^{m_1}}y_1c^{2^{m_2}}z_1$ is a reduced word in $\pi_1(S)$ with $p(y) = \widetilde{e}$, and $c^{2^{m_1}}$ is the first subword of the form $c^{2^t}$ of $y$, and $c^{2^{m_2}}$ is the last subword of the form $c^{2^t}$ of $y$ if $y$ has at least $2$ subwords of the form $c^{2^t}$ or $c^{2^{m_2}} = e$ otherwise. Because $y \neq e$ and $p(y) = \widetilde{e}$ , $2^n$ must divide $2^{m_1}$, and $m_1 \geq n$. Similarly, if $c^{2^{m_2}} \neq e$, then we also have $m_2 \geq n$.

Now we have $e = g^{-1}a^{-1}x_1c^{2^{m_1}}y_1c^{2^{m_2}}z_1gb^{-1}$, where the words $x_1$ and $z_1$ have no letter $c$. Because sum of the exponents of any subwords of the form $c^t$ in $a$ and in $g^{-1} $ is less than $2^n$, $c^{2^{m_1}}$ will never be cancelled by $g^{-1}a^{-1}x_1$. If $c^{m_2} \neq e$, then, by a similar argument, $c^{2^{m_2}}$ also cannot be cancelled by $z_1gb^{-1}$. 

If, however, $c^{m_2}  = e$, then our previous identity must have the form $e = g^{-1}a^{-1}x_1c^{2^{m_1}}y_1gb^{-1}$, and $x_1$ and $y_1$ contain no $c$ letter. Now consider the abelianization of both sides, we got: $e = \mathbf{Ab}(a)^{-1}\mathbf{Ab}(x_1)\mathbf{Ab}(c^{2^{m_1}})\mathbf{Ab}(y_1)\mathbf{Ab}(b)^{-1}$ ($g$ and $g^{-1}$ will have the abelianizations cancelled in the identity).

Now if we compare the exponents with base $c$ in both sides, we will get $0 = -exp(a) + 2^{m_1} - exp(b)$, where $exp(t)$ is the exponent with base $c$ of the abelianization of the word $t$. By the choice of $n$, we have $2^{m_1} \geq 2^n > 2^{n-1} > exp(a) + exp(b)$. Contradiction.

As a result, the identity $e = g^{-1}a^{-1}x_1c^{2^{m_1}}y_1c^{2^{m_2}}z_1gb^{-1}$ cannot hold. Therefore, we finish our proof by contradiction.
\end{proof}

\newpage
\begin{theorem}
There is a strictly descending chain of ideals that haven't been considered in \cref{st3}. For surface with more than one boundary component, the intersection of these ideals is zero.
\end{theorem}

\begin{proof}
The proof follows from \cref{lm11}, \cref{lm12}, and \cref{lm13}, and \cref{thm5}
\end{proof}

\end{document}